\documentclass[12pt]{article}
\usepackage{savesym}

\usepackage[colorlinks=true]{hyperref}
\hypersetup{allcolors=[rgb]{0.1,0.1,0.4}}

\usepackage[matrix,arrow,curve,rotate]{xy}
\input{diagxy}

\def\pushright#1{{
   \parfillskip=0pt            
   \widowpenalty=10000         
   \displaywidowpenalty=10000  
   \finalhyphendemerits=0      
   \leavevmode                 
   \unskip                     
   \nobreak                    
   \hfil                       
   \penalty50                  
   \hskip.2em                  
   \null                       
   \hfill                      
   {#1}                        
   \par}}                      

\def\qEd{{\lower1 pt\hbox{\vbox{\hrule\hbox{\vrule\kern4 pt
    \vbox{\kern4 pt\hbox{\hss}\kern4 pt}\kern4 pt\vrule}\hrule}}}}
\def\qed{\pushright{\qEd}
    \penalty-700 \par\addvspace{\medskipamount}}

\usepackage{xspace,mathrsfs}

\newcommand{\vsubseteq}{\rotatebox[origin=c]{-90}{$\subseteq$}}

\newcommand{\inob}{\in}

\newcommand{\catu}{\mathbf{Cat}}
\newcommand{\spancat}{\mathbf{Span}}
\newcommand{\vcat}{\catv\mathbf{CAT}}
\newcommand{\lightcatu}{\mathbf{CAT}}

\newcommand{\nlightcat}{n\mathbf{CAT}}
\newcommand{\bigmoncatu}{\mathbf{MONCAT}}

\newcommand{\twocat}{2\underline{\mathbf{CAT}}}
\newcommand{\modcat}{\underline{\mathbf{CAT}}}

\newcommand{\modcatku}{\underline{\mathbf{CAT}}_{k}}
\newcommand{\iterfam}[1]{\Psi{#1}}
\newcommand{\iterunfam}[1]{\Psi_{\univv}{#1}}
\newcommand{\iterfamo}{\Psi}
\newcommand{\class}{\mathbf{Class}}

\newcommand{\ncat}{n\mathbf{Cat}}
\newcommand{\nmodcat}{n\underline{\mathbf{CAT}}}
\newcommand{\equivr}[1]{\mathrm{Eq({#1})}}
\newcommand{\hfinite}{\mathbf{HF}}
\newcommand{\setu}{\set_{\univv}}
\newcommand{\classu}{\class_{\univv}}
\newcommand{\aset}[1]{^{\circ}{#1}}
\newcommand{\subo}[1]{i_{{#1}}}
\newcommand{\subs}[1]{\mathrm{Sub}({#1})}
\newcommand{\subsstar}[1]{\mathrm{Sub}^{*}({#1})}

\newcommand{\univv}{\mathfrak{U}}
\newcommand{\univvs}{\mathfrak{\tilde{U}}}

\newcommand{\univp}{\mathfrak{U}'}

\newcommand{\yons}{\mathcal{Y}}
\newcommand{\yonset}{\beta}

\newcommand{\funccat}[2]{[{#1},{#2}]}
\newcommand{\catc}{\mathcal{C}}
\newcommand{\catd}{\mathcal{D}}
\newcommand{\catw}{\mathcal{W}}
\newcommand{\catv}{\mathcal{V}}
\newcommand{\catm}{\mathcal{M}}
\newcommand{\cate}{\mathcal{E}}
\newcommand{\setbr}[1]{\{{#1}\}}
\newcommand{\beunixt}{\hspace{0.5em}}
\newenvironment{spaceout}[1]{\begin{displaymath}\setlength{\extrarowheight}{3pt}\begin{array}{#1}}{\end{array}\setlength{\extrarowheight}{0pt}\end{displaymath} \noindent}
\newcommand{\betwixt }{\hspace{1em}}

\newcommand{\ob}{\ensuremath{\mathsf{ob}\ }}
\newcommand{\op}{{\mathsf{op}}}
\newcommand{\nats}{\ensuremath{\mathbb{N}}}
\newcommand{\set}{\mathbf{Set}}
\newcommand{\pset}{\mathcal{P}}
\newcommand{\eqdef}{\stackrel{\mbox{\rm {\tiny def}}}{=}}
\newcommand{\iffdef}{\stackrel{\mbox{\rm {\tiny def}}}{\Leftrightarrow}}

\newtheorem{definition}{Definition}

\newtheorem{proposition}{Proposition}
\newtheorem{propprop}[proposition]{Proposed Theorem}
\newcommand{\smallcati}{\mathbb{I}}

\newcommand{\smallpresh}{\mathbf{SP}(\catc)}
\newtheorem{theorem}[proposition]{Proposition}
\newtheorem{corollary}[proposition]{Corollary}
\title{Formulating categorical concepts with classes}
\author{Paul Blain Levy, University of Birmingham}

\savesymbol{square}
\usepackage{amssymb,stmaryrd,array,graphicx,mathtools}

\begin{document}
\maketitle

\begin{abstract}
  We examine the use of classes to formulate several categorical notions.  This leads to two proposals: an explicit structure for working with subobjects, and a hierarchy of $k$-classes.  We apply the latter to both ordinary and higher categories.
\end{abstract}

\bibliographystyle{alpha}
\section{Introduction}

The notion of ``class'' pervades category theory but its role is not always apparent.  This article brings together a number of common concepts that it affects, and proposes some ways of formulating them.  We begin in Section~\ref{sect:prelim} by reviewing the basic notions of universe, class and category, and taking note of encoding issues for quotients and tuples.  In Section~\ref{sect:subob} we look at the theory of subobjects, leading to a notion of ``well-powering'', an  explicit structure for well-powered categories.  In Section~\ref{sect:hier} we propose a hierarchy of $k$-classes that is useful for formulating the Yoneda lemma and several other constructions, including higher category theory.  To make the general framework more user-friendly, Section~\ref{sect:sbd} proposes a convention---inspired by \cite{Murfet:sizecat}---for indicating size restrictions.  We sum up in Section~\ref{sect:conclusion}.


Many foundational systems have been proposed to deal with size issues in category theory, see e.g.~\cite{EnayatGorbowMcKenzie:forays,Feferman:setcat,Muller:setsclasses,Shulman:setcat}; an extensive survey is given in~\cite{Shulman:setcat}.  But we shall use the conventional framework of ZFC with universes.


\paragraph{Related work.} Size issues have been widely discussed, e.g.\ in the textbooks~\cite{AdamekHerrlichStrecker:concretecat,MacLane:categoriesbook}.  Dowd~\cite{Dowd:hightypecat} considered categorical applications of a hierarchy of classes in an extended version of ZFC.   Categories of classes have been studied by the ``Algebraic set theory'' school, e.g.~\cite{JoyalMoerdijk:algset,AwodeyButzSimpsonStreicher:relclasses}, and functors on them by~\cite{AczelMendler:finalcoalg,AdamekMiliusVelebil:coalgclass}.

\section{Preliminaries} \label{sect:prelim}
\subsection{Universes} \label{sect:univ}

In many accounts of category theory, a category is taken to have a \emph{class}  of objects, and there is a category of all sets. However (as stated above) we are working in ZFC, so we cannot speak of classes.  Instead we define a category $\catc$ to consist of a \emph{set} $\ob \catc$ and a family
 of sets $(\catc(x,y))_{x,y \inob \catc}$ together with composition and identities. By Russell's Theorem, there is no category of all sets.  That is a problem, and the notion of a (Grothendieck) universe provides a way of dealing with it.
\begin{definition}
Let $\univv$ be a set.  A \emph{universe} is a set $\univv$ with the following properties.
   \begin{itemize}
    \item Any set in $\univv$ is a subset of $\univv$.
    \item $\emptyset \in \univv$.
    \item If $x,y \in \univv$ then $\{x,y\} \in \univv$.
    \item If $I$ is a set in $\univv$ and $(A_i)_{i \in I}$ is a family of
      sets in $\univv$ then $\bigcup_{i \in I} A_i \in \univv$.
    \item If $A$ is a set in $\univv$ then $\pset A\in \univv$.
    \end{itemize}
\end{definition}
The least universe is the set $\hfinite$ of hereditarily finite sets, which does not contain $\nats$.  All other universes do contain $\nats$, but it cannot be proved in ZFC that such universes exist (assuming ZFC consistent). 

\begin{definition}
 Let $\univv$ be a universe.
 \begin{itemize}
 \item A \emph{$\univv$-small set} is a set in $\univv$.
\item A \emph{$\univv$-class} is a subset of $\univv$.
 \end{itemize}
We write
\begin{itemize}
\item $\setu$ for the category of $\univv$-small sets and functions
\item $\classu$ for the category of $\univv$-classes and functions.
\end{itemize}
\end{definition}
Thus $\setu \subsetneqq \classu$.  In particular, $\univv$ itself is a $\univv$-class but not a $\univv$-set.

The construction $\univv \mapsto \setu$ is designed to serve as a kind of substitute for the category of all sets.  But the extent to which it succeeds depends on what we assume about the existence of universes.  To see why, consider the following statements:
\begin{proposition}\label{prop:invcat}
  In any category $\catc$, a morphism $f \colon A \to B$ has at most one inverse.
\end{proposition}
\begin{proposition}\label{prop:invsetu}
  Let $\univv$ be a universe.  In $\setu$, a morphism  $f \colon A \to B$ has at most one inverse.
\end{proposition}
\begin{proposition}\label{prop:invset}
  For any sets $A$ and $B$, a function $f \colon A \to B$ has at most one inverse.
\end{proposition}
Proposition~\ref{prop:invsetu} is an instance of Proposition~\ref{prop:invcat}, but Proposition~\ref{prop:invset} (though easy to prove directly) cannot be deduced from Proposition~\ref{prop:invsetu}, because there might be no universe containing $A$ and $B$.  So the construction $\univv \mapsto \setu$ fails in its task of serving as a substitute for the category of all sets.   To avoid such difficulties, Grothendieck and Verdier~\cite{GrothendieckVerdier:sga4start} proposed the
  \emph{Universe Axiom}: every
  set belongs to a universe.
  Assuming this axiom allows us to deduce
  Proposition~\ref{prop:invset} from
  Proposition~\ref{prop:invsetu}.  (Even so, there remains a mismatch between the construction $\univv \mapsto \setu$ and the desired category of all sets.  See the discussion of reflection principles in~\cite{Shulman:setcat}.) 

This article is written both for people who assume the Universe Axiom and for those who do not.  Note that the books~\cite{AdamekHerrlichStrecker:concretecat,MacLane:categoriesbook} assume just one universe containing $\nats$.  

Henceforth, let $\univv$ be a universe.  We usually leave $\univv$ implicit, e.g.\ saying ``small'' for $\univv$-small, ``class'' for $\univv$-class, $\set$ for $\setu$, and $\class$ for $\classu$.

A set is \emph{essentially small} when it is isomorphic to a small set.  Essential smallness may seem a  more attractive notion than smallness, but there is no category of all essentially small sets.  For example,   ``essentially $\hfinite$-small'' means finite, and there is no category of all finite sets.  

\subsection{Small, light and moderate categories}

We consider the relationships between categories and $\univv$.
\begin{definition} \label{def:smalllightmodcat}
A category $\catc$ is
\begin{itemize}
\item \emph{small} when $\ob \catc$ and all
    the homsets are small
\item \emph{light} when $\ob \catc$ is a class and all the homsets are small
\item \emph{moderate} when $\ob \catc$ and all the homsets are classes~\cite{Shulman:exactcomp,Street:notionsoftop}.
\end{itemize}
We write
\begin{itemize}
\item $\catu$ for the 2-category of small categories
\item $\lightcatu$ for the 2-category of light categories
\item $\modcat$ for the 2-category of moderate categories.
\end{itemize}
\end{definition}
Thus $\catu \subsetneqq \lightcatu \subsetneqq \modcat$.  Here are some examples:
\begin{enumerate}
\item The category $\set_{\hfinite}$ is small, assuming $\nats \in \univv$.
\item The category  $\set$, and the category $\mathbf{Rel}$ of small sets and relations, are light but not small.
\item For sets $A,B$ a \emph{multirelation} 
  \begin{math}
    \xymatrix{
{A} \ar[r]|*=0{\shortmid}^{p} & {B}
}
  \end{math}
 is a family of cardinals $(p_{a,b})_{a \in A, b \in B}$.  The identity multirelation on a set $A$ is
 is given at $a,a' \in A$ by $1$ if $a=a'$ and $0$ otherwise; the composite of multirelations 
  \begin{math}
    \xymatrix{
{A} \ar[r]|*=0{\shortmid}^{p} & {B} \ar[r]|*=0{\shortmid}^{q} & {C}
}
  \end{math} 
is given at $a \in A, c \in C$ by $\sum_{b \in B} p_{a,b}q_{b,c}$.  The category $\mathbf{Multirel}$ of small sets and small multirelations (i.e.\ multirelations consisting of small cardinals) is moderate but not light.
\item The category $\class$ and the functor category $[\set,\set]$ are not moderate.
\end{enumerate}
Note that $\catu$ is cartesian closed but $\lightcatu$ and $\modcat$ are not.   If we want a cartesian closed 2-category containing $\set$, we may use $\catu_{\univp}$ for some universe $\univp$ larger than $\univv$, provided it exists (an instance of the Universe Axiom). 

The following conditions, weaker than lightness, are sometimes considered.
\begin{itemize}
\item A category is \emph{locally small} when all its homsets are
  small.  Thus a light category is one that is both moderate and
  locally small.  Some theorems about light categories, such as the
  adjoint functor theorems, hold more generally for locally small
  categories.  But it is hard to find natural examples of locally
  small categories that are not light, other than ones arising from a
  preordered set.  Moreover, there is no 2-category of all locally
  small categories.

\item A category is \emph{essentially light} when it is equivalent to a
  light category. For example, given a light category $\catc$, let $\smallpresh$ be the full subcategory of $[\catc^{\op},\set]$ on presheaves that are ``small'', i.e.\ isomorphic to the colimit of some small diagram of representables~\cite{DayLack:limitssmallfunc}.  This category is neither moderate nor locally small, but it is essentially light.  Moreover, via the Yoneda embedding, it is a free cocompletion of $\catc$.  So we might wish to view the construction $\catc \mapsto \smallpresh$ as a reflection of a 2-category of categories into a 2-category of cocomplete categories.  But we cannot, as there is no 2-category of essentially light categories.
\end{itemize}

\subsection{Quotient and tuple classes} \label{sect:qtup}

When working with classes, one must take care with the encoding of quotients and tuples.  
\begin{itemize}
\item  For an equivalence relation $R$ on a class $A$, the usual quotient $A/R$ is not a class. 
In order to form quotient classes, we first associate to every inhabited class $X$ an element $\theta X \in \univv$, in such a way that $\theta X \not= \theta Y$ whenever $X \cap Y =\emptyset$.  The following are two ways of doing this.
  \begin{enumerate}
  \item Let $\theta$ be a choice function on $\univv$, so $\theta X \in X$.
 \item Scott's trick: let $\theta X$ be the set of elements of $X$ of least rank. 
  \end{enumerate}
Now we set
  \begin{eqnarray*}
A \,/^*\, R & \eqdef & \setbr{[x]_R^{*} \mid x \in A}
\end{eqnarray*}
where $[x]_{R}^{*} \eqdef \theta \setbr{y \in A \mid (x,y) \in R}$.  The $^{*}$ superscript indicates a non-standard encoding.
\item For classes $A$ and $B$, the Kuratowski pair $(A,B) \eqdef \setbr{\setbr{A}, \setbr{A,B}}$ is not a class.  In order to form pair classes, following e.g.~\cite{AdamekHerrlichStrecker:concretecat}, we may use the encoding 
  \begin{eqnarray*}
    (A,B)^1 & \eqdef & \setbr{(0,x) \mid x \in A}\ \cup\ \setbr{(1,y) \mid y \in B}
  \end{eqnarray*}
Likewise, for a class $I$, we may encode an $I$-indexed tuple of classes by
\begin{eqnarray*}
  (A_i)^{1}_{i \in I} & \eqdef & \setbr{(i,x) \mid i \in I, x \in A_i} 
\end{eqnarray*}
A moderate category, encoded in this way, is a class.
\end{itemize}

\section{Subobjects} \label{sect:subob}

The theory of subobjects is commonly formulated using quotient classes.  We shall present this formulation and then propose a slight change.  The theory arises in the following situation.
\begin{definition}
   A wide subcategory $\catm$ of a category $\catc$ is \emph{mono-like} when
    \begin{itemize}
    \item  every $\catm$-morphism is monic in $\catc$
\item if a
  composite
  \begin{math}
    \xymatrix{ a \ar[r]^{f} & b \ar[r]^{g} & c }
  \end{math} is in $\catm$, then so is $f$.
    \end{itemize}
\end{definition}
Thus, in particular, all split monos are in $\catm$. Given a light category $\catc$ with a mono-like subcategory $\catm$, we proceed as follows.
\begin{definition}
Let $c \in \catc$.
  \begin{enumerate}
  \item  We form the class $\catm/c$ of pairs $(x,f)$ consisting of  $x \in \catc$ and an $\catm$-morphism $f \colon x \to c$, preordered as follows:  $(x,f) \sqsubseteq (y,g)$ when there is a morphism $h \colon x \to y$, necessarily unique and in $\catm$, making
     \begin{math}
       \xymatrix{
  x \ar[dr]_{f} \ar[r]^-{h} & y \ar[d]^{g} \\
 & c
}
     \end{math} commute.
\item  When   $(x,f)$ and $(y,g)$ are mutually related, the two mediating maps are mutually inverse, so we write $(x,f) \cong (y,g)$.
  \end{enumerate}
\end{definition}
Our task is to represent these pairs $(x,f)$ modulo $(\cong)$.  A commonly used formulation is as follows.
   \begin{definition}\label{def:wpowered}
For $c \in \catc$, the class of \emph{$\catm$-subobjects} of $c$ is 
    \begin{eqnarray*}
      \subsstar{c} & \eqdef & (\catm/c)\; /^{*} \;(\cong)
    \end{eqnarray*}
ordered as follows:
\begin{spaceout}{rclcrcl}
[(x,f)]^{*}_{\cong} & \leqslant & [(y,g)]^{*}_{\cong} &\betwixt \iffdef \betwixt\  & (x,f) & \sqsubseteq & (y,g)
\end{spaceout}%
We say $\catc$ is \emph{$\catm$-well-powered} when $\subsstar{c}$ is small for all $c \in \catc$.
  \end{definition}
Note that the isomorphic alternative
 \begin{eqnarray*}
      \subs{c} & \eqdef & (\catm/c)\;/ \;(\cong)
    \end{eqnarray*}
would be unsuitable.  For example, $\subsstar{1}$ is a subobject classifier in $\set$, but $\subs{1}$ is not, since it is not even an object.

Definition~\ref{def:wpowered} ingeniously makes  $\catm$-well-poweredness into a \emph{property} of $\catc$ and $\catm$, with no need for additional data.  But we propose a slight reformulation that, while it does require additional data, avoids the need for quotient classes.
\begin{definition}
Let $R$ be an equivalence relation on a set $A$.  A \emph{family of unique $R$-representatives} for $A$ is a set $I$ and family $(a_i)_{i \in I}$ of elements of $A$, such that, for every $a \in A$, there is a unique $i \in I$ for which $(a,a_i) \in R$. 
\end{definition}

 \,
\begin{definition} \label{def:wpowering}
An \emph{$\catm$-well-powering} $\catw$ assigns to each $c \in \catc$ a small family of unique $(\cong)$-representatives for $\catm/c$.  We write
\begin{eqnarray*}
  \catw \colon c & \mapsto & (\aset{U},\subo{U})_{U \in \subs{c}}
\end{eqnarray*}
We call $\subs{c}$ the set of \emph{$\catw$-subobject-indices} of $c$, ordered as follows.
\begin{spaceout}{rclcrcl}
  U & \leqslant & V & \beunixt\iffdef\beunixt\ & (\aset{U},\subo{U}) & \sqsubseteq  & (\aset{V},\subo{V})
\end{spaceout}%
\end{definition}

\begin{proposition}\hfill
  \begin{enumerate}
  \item \label{item:wellpowered} There is an $\catm$-well-powering $\catw$ iff $\catm$ is well-powered.  Moreover, $\catw$ is unique up to unique isomorphism.
  \item \label{item:recoverm} $\catm$ is determined by $\catw$.  Explicitly, a $\catc$-morphism $b \to c$ is in $\catm$ iff it is of the form
    \begin{math}
      \xymatrix{
 b  \ar[r]^-{g} & \aset{U}  \ar[r]^-{\subo{U}} & c
}
    \end{math} for a (necessarily unique) pair $(U,g)$ consisting of $U \in \subs{c}$ and an isomorphism $g \colon b \cong {} \aset{U}$.
  \end{enumerate}
\end{proposition}
\begin{proof}
  (\ref{item:wellpowered})($\Leftarrow$) is by the Axiom of Choice and
  the rest is straightforward. \qed
\end{proof}

In many cases there is a canonical $\catm$-well-powering.  For example, a well-powering of $\set$ for injections is given by
 \begin{eqnarray*}
c & \mapsto & (U,i_{U})_{U \in \pset c}
\end{eqnarray*}
where $i_{U} \colon U \to c$ is the inclusion $x \mapsto x$.  Thus the subobject-indices of $c$ are subsets and ordered by inclusion, rather than sets of (set, injection) pairs.  

There is an evident dual notion of an \emph{$\cate$-co-well-powering} of $\catc$, where $\cate$ is an \emph{epi-like} subcategory.
Again, in many cases there is a canonical one.  For example, a co-well-powering of $\set$ for surjections is given by 
\begin{eqnarray*}
  c & \mapsto & (c/r, p_{r})_{r \in \equivr{c}}
\end{eqnarray*}
where $\equivr{c}$ is the set of equivalence relations on $c$, and $p_{r} \colon c \to c/r$ sends $x \mapsto [x]_{r}$.  Thus the quotient-indices of $c$ are equivalence relations and ordered by inclusion, rather than sets of (set, surjection) pairs.  

The convenience of these notions for categorical writing is illustrated in~\cite{Levy:finalcorecur} (though they are not explicitly formulated there).  The content of that paper is presented both in the general setting of a category with a factorization system and in special cases involving subsets and equivalence relations.  The latter cases are instances of the former---precisely, not just up to isomorphism---because of the use of subobject-indices and quotient-indices.

For another example where a family of unique representatives is used instead of a quotient class, see~\cite[Theorem 3.24]{AdamekMiliusMossSousa:wellpointed}.   

\section{A hierarchy of classes} \label{sect:hier}

\subsection{The target of the Yoneda lemma}

In Section~\ref{sect:kclass} we shall introduce a new notion of \emph{$k$-class}.  To motivate this,  we first discuss the Yoneda lemma.  For a light category $\catc$, we define in the usual way
\begin{itemize}
\item a functor $\yons \colon \catc \to [\catc^{\op},\set]$
\item for $c \in \catc$ and $F \colon \catc^{\op} \to \set$ and $x \in Fc$, a natural transformation $\yonset_{c,F}(x) \colon \yons{c} \to F$.
\end{itemize}
Here is our first attempt to state the Yoneda lemma:
  \begin{proposition} \label{prop:yonshort} Let $\catc$ be a light category.
    Then $\yonset_{c,F}$ is a bijection
    $Fc \cong [\catc^{\op},\set](\yons{c},F)$, natural in $c$ and
    $F$.
  \end{proposition}
Expanding this statement reveals a problem.
  \begin{proposition} Let $\catc$ be a light category.  Then we
    have a natural isomorphism
    \begin{displaymath}
      \xymatrix{
        \catc^{\op} \times [\catc^{\op},\set] \ar[d]_{\mathsf{app}} \ar[rrr]^-{\yons^{\op} \times [\catc^{\op},\set]} \ar@{}[drrr]|{\stackrel{\yonset}{\Longrightarrow}} & & & [\catc^{\op},\set]^{\op} \times  [\catc^{\op},\set] \ar[d]^{\mathsf{hom}} \\
        \set \ar@{^{(}->}[rrr] & & & ?
      }
    \end{displaymath}
  \end{proposition}
  What should the target category be?

Before answering this, let us note that unpacking Proposition~\ref{prop:yonshort} gives a collection of statements that do not mention $\funccat{\catc^{\op}}{\set}$.  For example, the claim that $\yonset_{c,F}$ is natural in $F$ means that for any natural transformation 
\begin{math}
  \xymatrix{
 \catc^{\op} \ar@/^/[r]^{F} \ar@/_/[r]_{G} \ar@{}[r]|{\alpha \Downarrow} & \set
}
\end{math}
and $c \in \catc$ and $x \in Fc$, the composite 
\begin{math}
  \xymatrix{
 \yons{c} \ar[r]^-{\yonset_{c,F}(x)} & F \ar[r]^-{\alpha} & G
}
\end{math}
is $\yonset_{c,G}(\alpha_c x)$.    So we might view Proposition~\ref{prop:yonshort} as a mere figure of speech, summarizing this collection of statements.  But we are going to take it literally.  So we need a target category.

One option is to use $\set_{\univp}$, where
  $\univp$ is a universe greater than $\univv$ such that $\catc$ is
  $\univp$-small.  But while such a universe is guaranteed to exist if 
  the Universe Axiom is assumed, it is hardly relevant to the Yoneda lemma.
  After all, each homset of $[\catc^{\op},\set]$ is just a set of
  classes.  This suggests using a smaller category than $\set_{\univp}$, one that is not cartesian closed.

\subsection{$k$-classes} \label{sect:kclass}

To summarize our situation, we want to formulate the Yoneda lemma for a light category $\catc$ without mentioning a larger universe. Let us say that our target category will be the category of ``2-classes''.  What is a $2$-class?  

We certainly want every set of classes to be a 2-class. So, noting that $(\pset^{k}\univv)_{k \in \nats}$ is an increasing chain, it is reasonable to define a \emph{$k$-class} to be an element of $\pset^{k} \univv$.  But if we adopt this definition, then a binary product of $2$-classes is not a $2$-class, because a pair $(A,B)$ of classes is not a class.  Using the pair encoding $(A,B)^{1}$ from Section~\ref{sect:qtup} would only postpone the problem: a pair $(A,B)^{1}$ of 2-classes is not a 2-class.  

One solution would be to adopt a different encoding for each level:
\begin{eqnarray*}
  (A,B)^{0} & \eqdef & (A,B) \\
 (A_i)^{0}_{i \in I} & \eqdef &  (A_i)_{i \in I} \\
  (A,B)^{k+1} & \eqdef & \setbr{(0,x)^{k} \mid x \in A}\ \cup\ \setbr{(1,y)^{k} \mid y \in B} \\
  (A_i)^{k+1}_{i \in I} & \eqdef & \setbr{(i,x)^{k} \mid i \in I, x \in A_i} 
\end{eqnarray*}
Then $\pset^{k} \univv$ is closed under $(-,-)^{k}$ and, for $I \in \pset^{k}\univv$, under $(-)_{i \in I}^{k}$.  But having to continually distinguish all these encodings would be inconvenient.  Scott and McCarty~\cite{ScottMcCarty:ordpair} solved this problem by proving\footnote{This is a theorem of NBG class theory. 
} that there is a unique binary operation $(-,-)^{*}$ satisfying
\begin{eqnarray} \label{eqn:scottmc}
   (A,B)^{*} & = & \setbr{(0,x)^{*} \mid x \in A}\ \cup\ \setbr{(1,y)^{*} \mid y \in B} 
\end{eqnarray}
It is an ordered pair operation and every universe is closed under it.  They likewise encode indexed tuples:
\begin{eqnarray*}
   (A_i)^{*}_{i \in I} & \eqdef & \setbr{(i,x)^{*} \mid i \in I, x \in A_i} 
\end{eqnarray*}
It follows that $\pset^{k}\univv$ is closed under $(-,-)^{*}$ and, for $I \in \pset^{k}\univv$, under $(-)^{*}_{i \in I}$.

This is an ingenious solution, but we propose a different approach that avoids the need to replace the Kuratowski encoding. It uses the following construction.
\begin{definition}
  Let $A$ be a set of sets.  We inductively define the set $\iterunfam{A}$, or $\iterfam{A}$ for short, as follows.
\begin{itemize}
\item If $x \in \univv$, then $x \in \iterfam{A}$.
\item If $I \in A$, then $I \in \iterfam{A}$.
 \item If $x,y \in \iterfam{A}$, then $(x,y) \in \iterfam{A}$.
\item If $I \in A$, and $x_i \in \iterfam{A}$ for all $i \in I$, then $(x_i)_{i \in I} \in \iterfam{A}$.
\end{itemize}
Concisely, $\iterfam{A}$ is the least prefixpoint of
 \begin{math}
 X \beunixt \mapsto \beunixt  \univv \beunixt \cup \beunixt A\beunixt \cup\beunixt \cup \beunixt X \times X \beunixt \cup \beunixt\bigcup_{I \in A}   X^I
  \end{math}.
\end{definition}
Thus any element of $\iterfam{A}$ can be represented (not necessarily uniquely) by a well-founded tree that has
\begin{itemize}
\item leaves labelled by some $x \in \univv$
\item leaves labelled by some $I \in A$
\item binary nodes
\item and nodes labelled by some $I \in A$, which are $I$-ary.
\end{itemize}
Our key observation is that $\pset\iterfam{A}$ is closed under several constructions.  
\begin{proposition}\label{prop:psetiter}
  Let $A$ be a set of sets.
  \begin{enumerate}
\item If $B$ and $C$ are subsets of $\iterfam{A}$, then so are
    \begin{eqnarray*}
      B + C & \eqdef & \setbr{(0,b) \mid b \in B}\  \cup\  \setbr{(1,c) \mid c \in C} \\
\text{and } \beunixt     B \times C & \eqdef & \setbr{(b,c) \mid b \in B, c \in C}
    \end{eqnarray*}
  \item If $B$, and $C_b$ for all $b \in B$, are subsets of $\iterfam{A}$, then so is
    \begin{eqnarray*}
      \sum_{b \in B}C_b & \eqdef & \setbr{(b,c) \mid b \in B, c \in C_b}
    \end{eqnarray*}
  \item Let $I \in A$.  If $B_i$, for all $i \in I$, is a subset of $\iterfam{A}$, then so is
    \begin{eqnarray*}
      \prod_{i \in I}B_i & \eqdef & \setbr{(b_i)_{i \in I} \mid \forall i \in I.\, b_i \in B_i}
    \end{eqnarray*}
  \end{enumerate}
\end{proposition}
Let us write $\univvs$ for the set of sets in $\univv$.  (In ZFC, everything is a set so $\univvs=\univv$.  But in a set theory that allows urelements, $\univvs$ might be a proper subset of $\univv$.)  Since $\iterfamo$ and $\pset$ are monotone and $\iterfam{\univvs} = \univv$, we have
\begin{displaymath}
  \begin{array}{ccccc}
    \univvs & & & \subseteq  & \univv \\
\vsubseteq & &  & &  \vsubseteq \\
    \pset \iterfam{\univvs} & = & \pset \univv & \subseteq & \iterfamo \pset \univv \\
\vsubseteq & & \vsubseteq & &  \vsubseteq \\
    \pset \iterfamo \pset \iterfamo \univvs & = & \pset \iterfamo \pset \univv & \subseteq & \iterfamo \pset \iterfamo \pset \univv \\
\vsubseteq & & \vsubseteq & &  \vsubseteq \\
 \pset \iterfamo \pset \iterfamo \pset \iterfamo \univvs & = & \pset \iterfamo \pset \iterfamo \pset \univv & \subseteq & \iterfamo \pset \iterfamo \pset \iterfamo \pset \univv \\
\vsubseteq & & \vsubseteq & &  \vsubseteq \\
\vdots & & \vdots & & \vdots
  \end{array}
\end{displaymath}
This suggests the following definition.
\begin{definition}\hfill
  \begin{enumerate}
 \item A \emph{$(\univv,k)$-entity}, or \emph{$k$-entity} for short, 
    is an element of $(\iterfamo\univvs)^{k}\univv$.
  \item A \emph{$(\univv,k)$-class}, or \emph{$k$-class} for short, 
    is an element of $(\pset\iterfamo)^{k}\univvs$.  
\item The category of
    $k$-classes is called $\class_{k}$.
  \end{enumerate}
\end{definition}
Thus ``$0$-class'' means small set and ``$(k+1)$-class'' means set of $k$-entities.  Moreover, every $k$-class is a $k$-entity.  

Proposition~\ref{prop:psetiter} gives the following ways of constructing $k$-classes.  For $k=0$ we read ``$k-1$'' as 0. 
\begin{theorem} \hfill
  \begin{enumerate}
  \item If $B$ and $C$ are $k$-classes, then so are $B + C$ and $B \times C$.
  \item If $B$, and $C_b$ for all $b \in B$, are $k$-classes, then so is $\sum_{b \in B}C_i$. 
  \item Let $I$ be a $(k-1)$-class.  If $B_i$, for all $i \in I$, is a $k$-class, then so is $\prod_{i \in I}B_i$. 
  \end{enumerate}
\end{theorem}

\noindent\textbf{Remark}   In view of the \emph{Ackermann coding} $\nats \cong \hfinite$, perhaps $(\hfinite,k)$-classes might constitute a convenient model of higher-order arithmetic, cf.~\cite{KayeWong:arithset}.



\subsection{$k$-moderate categories}

We shall see that $k$-classes provide useful relationships between categories and $\univv$.
\begin{definition}\label{def:modcat}
  A category $\catc$ is \emph{$k$-moderate} when $\ob \catc$ and all the homsets are  $k$-classes.
\end{definition}
Thus ``0-moderate'' means small, and ``$(k+1)$-moderate'' means that all objects and morphisms are $k$-entities. 

Proposition~\ref{prop:psetiter} implies the following.
\begin{proposition} \label{prop:funccat}
  A functor category $[\catc,\catd]$ is
  \begin{itemize}
  \item small if $\catc$ and $\catd$ are small
  \item light if $\catc$ is small and $\catd$ light
  \item $k$-moderate if $\catc$ is $(k-1)$-moderate and $\catd$ is $k$-moderate.
  \end{itemize}
\end{proposition}

\begin{corollary}
 The category $\class_k$ is $(k+1)$-moderate.
\end{corollary}
If $\catc$ is light (hence 1-moderate), then  $[\catc,\set]$ is $2$-moderate, by Proposition~\ref{prop:funccat}.  So we can formulate the Yoneda lemma as follows.
\begin{proposition} Let $\catc$ be a light category.  Then we have a natural isomorphism
\begin{displaymath}
  \xymatrix{
 \catc^{\op} \times [\catc^{\op},\set] \ar[d]_{\mathsf{app}} \ar[rrr]^-{\yons \times [\catc^{\op},\set]} \ar@{}[drrr]|{\stackrel{\yonset}{\Longrightarrow}} & & & [\catc^{\op},\set]^{\op} \times  [\catc^{\op},\set] \ar[d]^{\mathsf{hom}} \\
 \set \ar@{^{(}->}[rrr] & & & \class_{2}
}
\end{displaymath}
\end{proposition}
Note, by the way, the requirement for $\catc$ to be light, i.e.\ both moderate and locally small.  The statement would not make sense if we weakened the moderateness assumption to essential moderateness, or the local smallness assumption to local essential smallness.

\subsection{Higher categories}

Let us now consider
\begin{itemize}
\item the 2-category $\catu$ of small categories
\item the 2-category $\lightcatu$ of light categories
\item the 2-category $\modcatku$ of $k$-moderate categories.
\end{itemize}
What is the relationship between these 2-categories and $\univv$?  In order to answer this question, let us formulate, more generally, relationships between $n$-categories and $\univv$.  

We fix $n$, where $n \in \nats \cup \setbr{\infty}$.  For $n = \infty$ we assume $\nats \in \univv$ (as there does not appear to be a reasonable notion of $\hfinite$-small $\infty$-category) and read ``$n+1$'' as $\infty$.   

An $n$-category (which in this article means \emph{weak} $n$-category) consists of two parts.  Firstly, a collection of \emph{$r$-homsets}, for $0 \leqslant r < n+1$.  More precisely we have
\begin{itemize}
\item the \emph{$0$-homset} $\catc()$, i.e.\ set of objects
\item for any $a_0, b_0 \in \catc()$, the \emph{$1$-homset} $ \catc(a_0,b_0)$
\item for any $a_0,b_0  \in \catc()$ and $a_1,b_1 \in \catc(a_0,b_0)$, the \emph{$2$-homset} $\catc(a_0,b_0;a_1,b_1)$
\item and so forth.
\end{itemize}
Secondly some structure, which we omit.  Many definitions have been proposed (see e.g.~\cite{Leinster:survey}) and we shall not adopt any particular one.  So the statements in this section are merely proposals that we expect to be true for any reasonable notion of (weak) $n$-category.  

We shall now define the properties displayed in Figure~\ref{fig:ncatconstr}.
\begin{figure} {
    \begin{tabular}{|l|l|l|l|l|l|l} \hline
      Small & & & & &  &  \\
      0-light & Light & 2-light & \ldots & $(n+1)$-light
                       &  \\
      0-moderate & & & & Moderate & 2-moderate & \ldots
      \\ \hline
    \end{tabular} }
   \caption{Properties of an $n$-category, in order of increasing liberality}
  \label{fig:ncatconstr}
\end{figure}
In so doing we generalize Definitions~\ref{def:smalllightmodcat} and~\ref{def:modcat}. 
\begin{definition}
Let $\catc$ be an $n$-category.
  \begin{enumerate}
  \item We say $\catc$ is \emph{small} when, for  $0 \leqslant r < n+1$, each $r$-homset is small.
 \item Let $0 \leqslant k \leqslant n+1$.  We say $\catc$ is \emph{$k$-light} when
   \begin{itemize}
   \item for $0 \leqslant r < k$, each $r$-homset is a class
  \item for $k \leqslant r < n+1$, each $r$-homset is small.
   \end{itemize}
 \item Let $k \in \nats$.  We say $\catc$ is \emph{$k$-moderate} when, for  $0 \leqslant r < n+1$, each $r$-homset is a $k$-class.
\end{enumerate}
As usual the ``$1$-'' prefix may be omitted.
\end{definition}
Thus ``0-moderate'' means small, and ``$(k+1)$-moderate'' means that, for  $0 \leqslant r < n+1$, all $r$-cells are $k$-entities. 

We generalize Proposition~\ref{prop:funccat} as follows.

 \begin{propprop}
For $n$-categories $\catc$ and $\catd$, the functor $n$-category $[\catc,\catd]$ is 
  \begin{itemize}
  \item small if $\catc$ and $\catd$ are small
  \item $k$-light if $\catc$ is small and $\catd$ is $k$-light
  \item $k$-moderate if $\catc$ is $(k-1)$-moderate and $\catd$ is $k$-moderate.
  \end{itemize}
\end{propprop}

\begin{propprop} \label{prop:highcat} \hfill
  \begin{enumerate}
  \item The $(n+1)$-category $\ncat$ of small $n$-categories is light. 
  \item The $(n+1)$-category $\nlightcat_{k}$ of $k$-light $n$-categories is 2-moderate.
  \item The $(n+1)$-category $\nmodcat_{k}$ of $k$-moderate $n$-categories is $(k+1)$-moderate.
  \end{enumerate}
\end{propprop}
The case $n=0$ of Proposed Theorem~\ref{prop:highcat} consists of familiar facts:
\begin{itemize}
\item $\set$ is light.
\item $\class$ is 2-moderate.
\item $\class_{k}$ is $(k+1)$-moderate.
\end{itemize}
The case $n=1$ answers our initial question:
\begin{itemize}
\item $\catu$ is light.
\item $\lightcatu$ is 2-moderate.
\item $\modcatku$ is $(k+1)$-moderate.
\end{itemize}
Another useful case, for finite $n$, is that $n\underline{\mathbf{CAT}} \eqdef n\underline{\mathbf{CAT}}_{n}$ is $(n+1)$-moderate. 

As for the notion of $k$-lightness, the following illustrates its significance.
\begin{propprop}  
  Let $\catc$ be an $n$-category.  Then the $(n+1)$-category $\spancat(\catc)$ is
  \begin{itemize}
  \item small if $\catc$ is small
  \item $(k+1)$-light if $\catc$ is $k$-light
  \item $k$-moderate if $\catc$ is $k$-moderate.
  \end{itemize}
\end{propprop}

\section{Standard By Default} \label{sect:sbd}

\vspace{1ex}

As we have seen, in certain situations where two or more universes are commonly used, one suffices.  This simplifies categorical writing: we can work with a single universe parameter and leave it implicit, as we have done.  Only when we genuinely want more than one, or to choose an appropriate one using the Universe Axiom, would we mention universes explicitly.

Nonetheless, our terminology is still too verbose.  Consider the following passage:
\begin{quotation}
  A light category $\catc$ consists of a class $\ob \catc$ and family of small sets $(\catc(a,b))_{a,b \inob \catc}$ with composition and identities. An example is the light category of small groups, which has all small limits. Another is given by the well-ordered class of small ordinals.   Any small poset or small monoid gives a small category, and any light category $\catc$ gives a 2-moderate category $[\catc^{\op},\set]$.

 Light categories form a 2-moderate 2-category.  There is also the 2-light 2-category of small sets and small spans.  Finally we may consider the light $(\infty,1)$-category\footnote{An $(\infty,n)$-category is an $\infty$-category where, for all $k>n$, the $k$-cells are weakly invertible.  Several definitions have been proposed; see e.g.~\cite{BergnerRezk:comparison}.} of small $\infty$-groupoids.  It contains the fundamental $\infty$-groupoid of every small topological space.
\end{quotation}
This passage illustrates the convention we have used so far, which may be called \emph{Unrestricted By Default}.  Every set, category etc.\ mentioned is unrestricted, unless we specify some relationship with $\univv$.  This convention has served us well during our exploration of such relationships.  But it is unsuitable for ordinary writing, where size issues are not the main subject and should obtrude as little as possible.

To resolve this situation, we introduce the following terminology.  
\begin{definition} \label{def:standard}
A mathematical entity is described as \emph{$\univv$-standard}, or \emph{standard} for short, according to the following rules.
  \begin{itemize}
  \item A set, monoid, topological space, poset, family\footnote{In the sense of a pair $(I,(a_i)_{i \in I})$, where $I$ is a set.}, graph\footnote{A graph (more precisely called a \emph{quiver}) consists of a set
     $V$ of vertices, a set $E$ of edges, and
    source and target functions $s,t \colon E \to V$.}, diagram\footnote{In the sense of a pair $(\smallcati,D \colon \smallcati \to \catc)$, where $\smallcati$ is a graph.}, cardinal, ordinal etc.\ is standard when it is small.
  \item A category, groupoid, multicategory, locally ordered category etc.\ is standard when it is
    light.
  \item For $2 \leqslant n < \infty$, an $n$-category is
    standard when it is $n$-moderate.
  \item An $\infty$-groupoid is standard when it is small.
\item An $(\infty,1)$-category is standard when it is light.
  \item For $2 \leqslant n < \infty$, an $(\infty,n)$-category is standard when it is $n$-moderate.
  \item A function, relation, subset, functor, natural transformation etc.\ is always standard.
  \end{itemize}
\end{definition}
Definition~\ref{def:standard} is open-ended and based purely on convenience. It gives rise to a \emph{Standard By Default} convention: every entity is assumed to be standard, unless specified otherwise.   If we want to say that a set is not assumed to be small, we describe it as ``unrestricted'' or ``large''.  If we want to say that a category is not assumed to be light, we describe it as ``unrestricted'' or ``heavy''.

Here is a Standard By Default translation of the above passage:
\begin{quotation}
  A category $\catc$ consists of a class $\ob \catc$ and family of sets $(\catc(a,b))_{a,b \inob \catc}$ with composition and identities. An example is the category of groups, which has all limits. Another is given by the well-ordered class of ordinals.   Any poset or monoid gives a small category, and any category $\catc$ gives a 2-moderate category $[\catc^{\op},\set]$.

 Categories form a 2-category.  There is also the 2-light 2-category of sets and spans.  Finally we may consider the $(\infty,1)$-category of $\infty$-groupoids.  It contains the fundamental $\infty$-groupoid of every topological space.
\end{quotation}
Arguably this is close to current practice and not too onerous.  But the problem remains of interfacing with ordinary writing about groups, topological spaces, ordinals, $\infty$-groupoids etc.  Such writing has no universe parameter and therefore uses the Unrestricted By Default convention. The clash of conventions must be handled carefully, whether or not the Universe Axiom is assumed.

We finish by using Standard By Default to easily formulate an example from~\cite{Shulman:setcat}.  For a monoidal category $\catv$, a \emph{$\catv$-enriched category} $\catc$ consists of a class $\ob \catc$ and a family of $\catv$-objects $(\catc(a,b))_{a,b\inob\catc}$ with composition and identities.  We write $\bigmoncatu$ for the 2-category of monoidal categories, and $\vcat$ for that of $\catv$-enriched categories.
\begin{proposition} The construction 
$\catv \mapsto \vcat$ is a 2-functor from $\bigmoncatu$ to $\twocat$.  
\end{proposition}

\section{Conclusion} \label{sect:conclusion}

Using families of representatives and $k$-classes, we have formulated several categorical concepts in a way that avoids the need for sophisticated encodings of quotients and tuples.  All our definitions and statements are given relative to at most one universe.  The Standard By Default convention makes this into a reasonably lightweight framework.

Our treatment is robust in the following sense.  ZFC assumes that everything is a set and $\in$-well-founded---the \emph{von Neumann assumptions}. Our formulations, unlike Scott's trick and Scott-McCarty pairing, do not rely on these assumptions.  So they are suitable for those who adopt a weaker set theory that, for example, may allow class-many urelements or Quine atoms\footnote{A Quine atom is a set $x$ that is equal to $\setbr{x}$.}.

The notion of \emph{2-class}, i.e.\ set of $1$-entities, has been especially useful.  We have made use of $\ncat_{2}$ and its $n=0$ case $\class_{2}$, but not of the fact that they are 3-moderate.  It would be interesting to know whether any 4-class, or the notions of 3-class or 2-entity, appear in a significant concept or theorem.

\paragraph*{Acknowledgements} I thank Ohad Kammar for helpful discussion. I also thank Eduardo Dubuc, Thomas Streicher and Richard Williamson for explaining a curious claim in~\cite[page 3]{GrothendieckVerdier:sga4start} that, for a small category $\catc$, the functor category $[\catc,\set]$ is neither moderate nor locally small. This arises from the practice of tagging every function with its domain and codomain, and likewise every functor and natural transformation. By not adopting that practice, the problem is avoided.


\begin{thebibliography}{AMMS13}

\bibitem[ABSS14]{AwodeyButzSimpsonStreicher:relclasses}
Steven Awodey, Carsten Butz, Alex Simpson, and Thomas Streicher.
\newblock Relating first-order set theories, toposes and categories of classes.
\newblock {\em Ann. Pure Appl. Logic}, 165(2):428--502, 2014.

\bibitem[AHS90]{AdamekHerrlichStrecker:concretecat}
J.~Ad\'{a}mek, H.~Herrlich, and G.~Strecker.
\newblock {\em Abstract and Concrete Categories---The Joy of Cats}.
\newblock Wiley, 1990.

\bibitem[AM89]{AczelMendler:finalcoalg}
P.~Aczel and P.~F. Mendler.
\newblock A final coalgebra theorem.
\newblock In D.~H. Pitt, D.~E. Rydeheard, P.~Dybjer, A.~M. Pitts, and
  A.~Poign{\'e}, editors, {\em Proc. of the Conference on Category Theory and
  Comp. Sci.}, volume 389 of {\em LNCS}, pages 357--365, Berlin, September
  1989. Springer.

\bibitem[AMMS13]{AdamekMiliusMossSousa:wellpointed}
Jir{\'{\i}} Ad{\'{a}}mek, Stefan Milius, Lawrence~S. Moss, and Lurdes Sousa.
\newblock Well-pointed coalgebras.
\newblock {\em Logical Methods in Computer Science}, 9(3), 2013.

\bibitem[AMV04]{AdamekMiliusVelebil:coalgclass}
J.~Ad{\'a}mek, S.~Milius, and J.~Velebil.
\newblock On coalgebra based on classes.
\newblock {\em Theor. Comput. Sci}, 316(1):3--23, 2004.

\bibitem[BR13]{BergnerRezk:comparison}
Julia~E. Bergner and Charles Rezk.
\newblock Comparison of models for {$(\infty, n)$-categories I}.
\newblock {\em Geometry and Topology}, 17(4):2163--2202, 2013.

\bibitem[DL07]{DayLack:limitssmallfunc}
Brian~J. Day and Stephen Lack.
\newblock Limits of small functors.
\newblock {\em J. Pure Appl. Algebra}, 210(3):651--663, 2007.

\bibitem[Dow93]{Dowd:hightypecat}
Martin Dowd.
\newblock Higher type categories.
\newblock {\em Mathematical Logic Quarterly}, 39:251--254, 1993.

\bibitem[EGM17]{EnayatGorbowMcKenzie:forays}
Ali Enayat, Paul Gorbow, and Zachiri McKenzie.
\newblock Feferman's forays into the foundations of category theory.
\newblock In G.~Jaeger and W.~Sieg, editors, {\em Feferman on Foundations:
  Logic, Mathematics and Philosophy}. Springer, 2017.

\bibitem[Fef69]{Feferman:setcat}
Solomon Feferman.
\newblock Set-theoretical foundations of category theory.
\newblock In {\em Reports of the Midwest Category Seminar, III}, pages
  201--247. Springer, 1969.

\bibitem[GV64]{GrothendieckVerdier:sga4start}
Alexander Grothendieck and J.~L. Verdier, editors.
\newblock {\em {S}{\'e}minaire de {G}{\'e}ometrie Alg{\'e}brique, {IV}}, number
  269 in Lecture Notes in Mathematics. Springer-Verlag, 1964.

\bibitem[JM95]{JoyalMoerdijk:algset}
Andr{\'e} Joyal and Ieke Moerdijk.
\newblock {\em Algebraic Set Theory}, volume 220 of {\em London Math.\ Society
  Lecture Note Series}.
\newblock Cambridge University Press, Cambridge, 1995.

\bibitem[KW07]{KayeWong:arithset}
Richard Kaye and Tin~Lok Wong.
\newblock On interpretations of arithmetic and set theory.
\newblock {\em Notre Dame Journal of Formal Logic}, 48(4):497--510, 2007.

\bibitem[Lei02]{Leinster:survey}
Tom Leinster.
\newblock A survey of definitions of n-category.
\newblock {\em Theory and Applications of Categories}, 10(1):1--70, 2002.

\bibitem[Lev15]{Levy:finalcorecur}
Paul~Blain Levy.
\newblock {Final Coalgebras from Corecursive Algebras}.
\newblock In Lawrence~S. Moss and Pawel Sobocinski, editors, {\em 6th
  Conference on Algebra and Coalgebra in Computer Science (CALCO 2015)},
  volume~35 of {\em Leibniz International Proceedings in Informatics (LIPIcs)},
  pages 221--237, Dagstuhl, Germany, 2015. Schloss Dagstuhl--Leibniz-Zentrum
  fuer Informatik.

\bibitem[ML71]{MacLane:categoriesbook}
Saunders Mac~Lane.
\newblock {\em Categories for the Working Mathematician}, volume~5 of {\em
  Graduate Texts in Mathematics}.
\newblock Springer, New York, 1971.

\bibitem[Mul01]{Muller:setsclasses}
F.~A. Muller.
\newblock Sets, classes and categories.
\newblock {\em British Journal for the Philosophy of Science}, 52(3):539--573,
  2001.

\bibitem[Mur06]{Murfet:sizecat}
Dan Murfet.
\newblock Foundations for category theory.
\newblock Available at \texttt{therisingsea.org}, 2006.

\bibitem[Shu08]{Shulman:setcat}
Michael Shulman.
\newblock Set theory for category theory, October~07 2008.
\newblock arXiv:0810.1279v2.

\bibitem[Shu12]{Shulman:exactcomp}
Michael Shulman.
\newblock Exact completions and small sheaves.
\newblock {\em Theory and Applications of Categories}, 27(7):97--173, 2012.

\bibitem[SM08]{ScottMcCarty:ordpair}
Dana~S. Scott and Dominic McCarty.
\newblock Reconsidering ordered pairs.
\newblock {\em Bulletin of Symbolic Logic}, 14(3):379--397, 2008.

\bibitem[Str81]{Street:notionsoftop}
R.~H. Street.
\newblock Notions of topos.
\newblock {\em Bull. Austr. Math. Soc.}, 23(2):199--207, 1981.

\end{thebibliography}
\end{document}